\documentclass{amsart}
\usepackage{amssymb}
\usepackage{amsmath}
\usepackage{latexsym}
\usepackage{eepic}
\usepackage{epsfig}
\usepackage{graphicx}
\usepackage{amscd}
\usepackage{eufrak}
\usepackage{mathrsfs}
\usepackage[english]{babel}

\DeclareMathAlphabet{\mathpzc}{OT1}{pzc}{m}{it}
\newtheorem{theorem}{Theorem}[section]
\newtheorem{theorem-definition}[theorem]{Theorem-Definition}
\newtheorem{lemma-definition}[theorem]{Lemma-Definition}
\newtheorem{definition-prop}[theorem]{Proposition-Definition}

\newtheorem{prop}[theorem]{Proposition}
\newtheorem{lemma}[theorem]{Lemma}
\newtheorem{cor}[theorem]{Corollary}
\newtheorem{definition}[theorem]{Definition}

\newtheorem{lem}[theorem]{Lemma}
\newtheorem{question}[theorem]{Question}

\newcommand{\LL}{\ensuremath{\mathbb{L}}}

\newcommand{\N}{\ensuremath{\mathbb{N}}}
\newcommand{\Z}{\ensuremath{\mathbb{Z}}}
\newcommand{\Q}{\ensuremath{\mathbb{Q}}}

\newcommand{\A}{\ensuremath{\mathbb{A}}}

\newcommand{\cX}{\ensuremath{\mathcal{X}}}

\newcommand{\cU}{\ensuremath{\mathscr{U}}}

\renewcommand{\A}{\ensuremath{\mathbb{A}}}

\renewcommand{\cU}{\ensuremath{\mathscr{U}}}

\newcommand{\la}{\ensuremath{\langle }}
\newcommand{\ra}{\ensuremath{\rangle }}

\newcommand{\Spec}{\ensuremath{\mathrm{Spec}\,}}
\newcommand{\Spf}{\ensuremath{\mathrm{Spf}\,}}

\numberwithin{equation}{section} \hyphenpenalty=6000
\begin{document}
\title{A note on motivic integration in mixed characteristic}


\author{Johannes Nicaise}

\address{KULeuven\\
Department of Mathematics\\ Celestijnenlaan 200B\\3001 Heverlee \\
Belgium} \email{johannes.nicaise@wis.kuleuven.be}

\author{Julien Sebag}

\address{Universit\'e Rennes 1\\
UFR math\'ematiques\\IRMAR\\
263 Avenue du G\'en\'eral Leclerc CS 74205\\
35042 Rennes Cedex\\
France}

\email{Julien.Sebag@univ-rennes1.fr} 

\begin{abstract}
We introduce a quotient of the Grothendieck ring of varieties by
identifying classes of universally homeomorphic varieties. We show
that the standard realization morphisms factor through this
quotient, and we argue that it is the correct value ring for the
theory of motivic integration on formal schemes and rigid
varieties in mixed characteristic.

The present note is an excerpt of a detailed survey paper which
will be published in the proceedings of the conference ``Motivic
integration and its interactions with model theory and
non-archimedean geometry'' (ICMS, 2008).
\end{abstract}

\maketitle

\section{Introduction}
The Grothendieck ring $K_0(Var_F)$ of varieties over a field $F$
arises naturally as the universal ring of additive and
multiplicative invariants of such varieties. Taking the class of a
variety in the Grothendieck ring is the most general way to
``measure the size'' of the variety. In recent years, the
Grothendieck ring of varieties has received much attention,
because of its role as value ring in several theories of motivic
integration.

In spite of this renewed interest, many basic questions on the
structure of the Grothendieck ring remain unanswered. The main
difficulty is that it may be very hard to decide whether two given
varieties have distinct classes in the Grothendieck ring. The
central question in this context is the one raised by Larsen and
Lunts in \cite[1.2]{Larsen-Lunts}.

\begin{question}[Larsen-Lunts]\label{que-lalu}
Let $F$ be a field, and let $X$ and $Y$ be  $F$-varieties such
that $[X]=[Y]$ in $K_0(Var_F)$. Is it true that $X$ and $Y$ are
piecewise isomorphic, i.e., that we can find an integer $n>0$,
subvarieties $X_1,\ldots,X_n$ of $X$ and subvarieties
$Y_1,\ldots,Y_n$ of $Y$ such that $X_i$ is $F$-isomorphic to $Y_i$
for every $i$ in $\{1,\ldots,n\}$?\end{question} This question has
been answered affirmatively in certain cases
 \cite{Liu-Sebag}\cite{Sebag-PAMS}, but it remains open in general. In the present note, we raise a different question.

\begin{question}\label{que-uhom}
Let $F$ be a field of characteristic $p>0$.  If $X$ and $Y$ are
$F$-varieties such that there exists a universal homeomorphism of
$F$-schemes $Y\rightarrow X$, is it true that $[X]=[Y]$ in
$K_0(Var_F)$?
\end{question}

It seems reasonable to expect that Question \ref{que-uhom} has a
negative answer, in general. In characteristic zero, universally
homeomorphic varieties are piecewise isomorphic (Proposition
\ref{prop-pis}). In positive characteristic, Question
\ref{que-uhom} is in some sense orthogonal to Question
\ref{que-lalu}, since there are examples of universally
homeomorphic $F$-varieties that are not piecewise isomorphic. The
most basic example is the following: assume that $F$ is imperfect,
and let $F'$ be a non-trivial finite purely inseparable extension
of $F$. Then the morphism $\Spec F'\rightarrow \Spec F$ is a
universal homeomorphism. We do not know if $[\Spec F']=[\Spec F]$
in $K_0(Var_F)$.

For every field $F$, we introduce a quotient $K_0\la Var_F\ra $ of
the Grothendieck ring of $F$-varieties by identifying classes of
universally homeomorphic varieties (Definition \ref{def-mod}). We
call this quotient the modified Grothendieck ring of
$F$-varieties. If $F$ has characteristic zero, then the projection
morphism
$$K_0(Var_F)\rightarrow K_0\la Var_F\ra$$ is an isomorphism
(Proposition \ref{prop-zero}). In any characteristic, there exists
a canonical isomorphism
$$K_0\la Var_F \ra\rightarrow K_0(ACF_F)$$ to the Grothendieck
ring of the theory $ACF_F$ of algebraically closed fields over $F$
(Proposition \ref{prop-acf}). We show that the standard
realization morphisms of the Grothendieck ring of varieties
(\'etale realization, Poincar\'e polynomial,$\ldots$) factor
through the modified Grothendieck ring $K_0\la Var_F\ra $
(Proposition \ref{prop-real}). In fact, this is what makes
Question \ref{que-uhom} hard to answer: we cannot use the standard
realization morphisms to distinguish classes of universally
homeomorphic varieties in the Grothendieck ring.

In Section \ref{sec-mixed}, we fill a gap in the proof of the
change of variables theorem for motivic integrals on formal
schemes in mixed characteristic \cite[8.0.5]{sebag1}. To this aim,
it is necessary to replace the Grothendieck ring of varieties by
the modified version introduced in Definition \ref{def-mod}. We
emphasize that this correction only affects the theory of motivic
integration in {\em mixed} characteristic; in equal characteristic
$p\geq 0$, the results in
 \cite[8.0.5]{sebag1} are valid as stated. The modification is
 harmless for the applications of the theory, because the
standard realization morphisms factor through the modified
Grothendieck ring. In Sections \ref{sec-rig} and \ref{sec-lit}, we
give a list of changes that should be made to the literature.

The present note is an excerpt of a detailed survey paper, which
will be published in the proceedings of the conference ``Motivic
integration and its interactions with model theory and
non-archimedean geometry'' (ICMS, 2008).
\subsection*{Acknowledgements}
The idea of identifying classes of universally homeomorphic
varieties in the Grothendieck ring has been suggested to us by A.
Chambert-Loir. We are very grateful to him for this suggestion,
and for many helpful discussions.

\subsection*{Notations}
We denote by $(\cdot)_{red}$ the functor from the category of
schemes to the category of reduced schemes that maps a scheme $X$
to its maximal reduced closed subscheme $X_{red}$. If $S$ is a
Noetherian scheme, then an $S$-variety is a reduced separated
$S$-scheme of finite type.

\section{The Grothendieck ring of varieties}\label{sec-k0}
Let $S$ be a Noetherian scheme.
 For the definition of the
Grothendieck ring of $S$-varieties $K_0(Var_S)$ and its
localization $\mathcal{M}_S$, we refer to \cite[2.1]{Ni-tracevar}.
We recall some of the main realization morphisms. For details, the
reader may consult \cite[\S\,2.1]{Ni-tracevar}.
\subsubsection*{Point counting} If $F$ is a finite field, then there exists a unique ring morphism
$$\sharp:K_0(Var_F)\rightarrow \Z$$ that maps $[X]$ to the
cardinality of the set $X(F)$ for every $F$-variety $X$.
\subsubsection*{Euler characteristic} If $F$ is a field, and $\ell$ a prime invertible in $F$, then
there exists a unique ring morphism
$$\chi_{top}:K_0(Var_F)\rightarrow \Z$$ that maps $[X]$ to the
$\ell$-adic Euler characteristic of $X$ for every $F$-variety $X$.
It localizes to a ring morphism
$$\chi_{top}:\mathcal{M}_F\rightarrow \Z$$
These morphisms are independent of $\ell$.

\subsubsection*{Galois realization}
Let $F$ be a field, and $\ell$ a prime invertible in $F$. We fix a
separable closure $F^s$ of $F$. We denote by $G_F$ the absolute
Galois group of $F$, and by $K_0(Rep_{G_F}\Q_\ell)$ the
Grothendieck ring of $\ell$-adic Galois representations of $F$.
There exists a unique ring morphism
$$Gal:K_0(Var_F)\rightarrow K_0(Rep_{G_F}\Q_\ell)$$ that maps $[X]$ to
$$\sum_{i=0}^{2\cdot \mathrm{dim}(X)}(-1)^i [H_c^i(X\times_F
F^s,\Q_\ell)]\quad \in K_0(Rep_{G_F}\Q_\ell)$$
 for every $F$-variety $X$. It localizes to a ring morphism
$$Gal:\mathcal{M}_{F}\rightarrow K_0(Rep_{G_F}\Q_\ell)$$
\subsubsection*{\'Etale realization}
Let $\ell$ be a prime, and $S$ a Noetherian $\Z[1/\ell]$-scheme.
There exists a unique ring morphism
$$\acute{e}t:K_0(Var_S)\rightarrow K_0(D^b_c(S,\Q_\ell))$$ that maps $[X]$ to
the class of $R(g_X)_!\Q_\ell$
 for every $S$-variety $X$, where we denote by $g_X:X\rightarrow S$ the structural morphism.
 It localizes to a ring morphism
$$\acute{e}t:\mathcal{M}_{S}\rightarrow K_0(D^b_c(S,\Q_\ell))$$
\subsubsection*{Poincar\'e realization}
Let $S$ be a Noetherian scheme. The Poincar\'e realization
$$P_S:K_0(Var_S)\rightarrow \mathcal{C}(S,\Z[T])$$ is defined in
\cite[8.12]{Ni-tracevar}. The target $\mathcal{C}(S,\Z[T])$ is the
ring of constructible functions on $S$ with values in $\Z[T]$
\cite[\S\,8.3]{Ni-tracevar}.
\section{The modified Grothendieck ring}
\subsection{Trivializing universal homeomorphisms}
Recall that a morphism of schemes $$f:X\rightarrow Y$$ is called a
{\em universal homeomorphism} if for every morphism of schemes
$Y'\rightarrow Y$, the morphism
$$f_{Y'}:X\times_Y Y'\rightarrow Y'$$ obtained from $f$ by base change is
a homeomorphism \cite[2.4.2]{ega4.2}. This property is obviously
stable under base change. If $f$ is of finite presentation, then
$f$ is a universal homeomorphism iff $f$ is finite, surjective,
and purely inseparable \cite[8.11.6]{ega4.3}. We call two schemes
$X$ and $Y$ {\em universally homeomorphic} if there exists a
universal homeomorphism from $X$ to $Y$ or from $Y$ to $X$. If $S$
is a scheme and $X$ and $Y$ are $S$-schemes, then we call $X$ and
$Y$ universally $S$-homeomorphic if there exists a universal
homeomorphism of $S$-schemes $X\rightarrow Y$ or $Y\rightarrow X$.

\begin{prop}\label{prop-pis}
If $f:X\rightarrow Y$ is a universal homeomorphism of finite type
between Noetherian $\Q$-schemes, then there exists a finite
partition $\{Y_1,\ldots,Y_r\}$ of $Y$ into locally closed subsets,
such that, if we endow $Y_i$ with its reduced induces structure,
the morphism $(X\times_Y Y_i)_{red}\rightarrow Y_i$ is an
isomorphism for each $i\in \{1,\ldots,r\}$.
\end{prop}
\begin{proof}
 Since
$f_{red}:X_{red}\rightarrow Y_{red}$ is still a universal
homeomorphism \cite[2.4.3(vi)]{ega4.2}, we may assume that $X$ and
$Y$ are reduced. By Noetherian induction, it is enough to find a
non-empty open subscheme $U$ of $Y$ such that $X\times_Y
U\rightarrow U$ is an isomorphism. In particular, we may assume
that $Y$ is irreducible. Then $X$ is irreducible, because it is
homeomorphic to $Y$. If we denote by $\eta_Y$ the generic point of
$Y$, then its inverse image in $X$ consists of a unique point
$\eta_X$, which is the generic point of $X$. The residue field
$\kappa(\eta_X)$ is a purely inseparable extension of the residue
field $\kappa(\eta_Y)$ of $\eta_Y$. Since these fields have
characteristic zero, we see that $f$ induces an isomorphism
$\kappa(\eta_X)\cong \kappa(\eta_Y)$, so that the restriction of
$f$ to some dense open subset of $X$ is an open immersion. This
concludes the proof.
\end{proof}

\begin{definition}\label{def-mod}
Let $S$ be a Noetherian scheme. We denote by $I^{uh}_S$ the ideal
in $K_0(Var_S)$ generated by elements of the form $[X]-[Y]$, where
$X$ and $Y$ are universally $S$-homeomorphic separated $S$-schemes
of finite type. We put $$K_0\langle Var_S\rangle
=K_0(Var_S)/I^{uh}_S$$ and we call this quotient the modified
Grothendieck ring of $S$-varieties.

For every separated $S$-scheme of finite type $Z$, we denote by
$\langle Z\rangle $ the image of $[Z]$ in $K_0\langle Var_S\rangle
$. We put $\widetilde{\LL}_S=\langle \A^1_S\rangle $, and we
denote by $\mathcal{M}^{mod}_S$ the localization of $K_0\langle
Var_S\rangle $ with respect to $\widetilde{\LL}_S$.

The dimensional completion $\widehat{\mathcal{M}}_S^{mod}$  is the
separated completion of $\mathcal{M}_S^{mod}$ with respect to the
descending filtration $F^{\bullet}\mathcal{M}_S^{mod}$, where for
every $i\in \Z$,  $F^i\mathcal{M}_S^{mod}$ is the subgroup of
$\mathcal{M}_S^{mod}$ generated by the elements of the form $\la X
\ra \widetilde{\LL}_S^j$ with $X$ a separated $S$-scheme of finite
type and $j$ an element of $\Z$ such that
$$\mathrm{dim}(X/S)+j\leq -i$$
Here $\mathrm{dim}(X/S)$ denotes the relative dimension of $X$
over $S$.
\end{definition}

If $S=\Spec A$ for some Noetherian ring $A$, then we also write
$I^{uh}_A$, $K_0\langle Var_A\rangle $, $\widetilde{\LL}_A$,
$\mathcal{M}^{mod}_A$ and $\widehat{\mathcal{M}}^{mod}_A$ instead
of $I^{uh}_S$, $K_0\langle Var_S\rangle $, $\widetilde{\LL}_S$,
$\mathcal{M}^{mod}_S$ and $\widehat{\mathcal{M}}^{mod}_S$.

Let $S$ be a Noetherian scheme, $X$ a separated $S$-scheme of
finite type, and $C$ constructible subset of $X$. We can write $C$
as a disjoint union of locally closed subsets $C_1,\ldots,C_r$ of
$X$. If we endow $C_i$ with its reduced induced structure, for
every $i$, then the class $[C]$ of $C$ in $K_0(Var_S)$ is defined
by
$$[C]=[C_1]+\ldots + [C_r]$$
This definition does not depend on the choice of $C_1,\ldots,C_r$.
We denote by $\langle C\rangle $ the image of $[C]$ in $K_0\langle
Var_S\rangle $.

\begin{prop}\label{prop-zero}
If $S$ is a Noetherian scheme over $\Q$, then $I^{uh}_S$ is the
zero ideal, and the projection
$$K_0(Var_S)\rightarrow K_0\langle Var_S\rangle $$ is an isomorphism.
\end{prop}
\begin{proof}
This follows immediately from Proposition \ref{prop-pis} and the
scissor relations in the Grothendieck ring $K_0(Var_S)$.
\end{proof}

If $S$ is not a $\Q$-scheme, we do not know if $I^{uh}_S$ is
different from zero. If $S$ is the spectrum of a field $F$ of
positive characteristic, then $I^{uh}_S$ is trivial if and only if
Question \ref{que-uhom} in the introduction has a positive answer.


\subsection{Base change and direct image}
The definitions of the modified Grothendieck ring $K_0\langle
Var_S\rangle $ and its localization $\mathcal{M}^{mod}_S$ are
compatible with base change and direct image. If $f:T\rightarrow
S$ is a morphism of Noetherian schemes, then there exists a unique
ring morphism
$$f^*:K_0\langle Var_S\rangle \rightarrow K_0\langle Var_T\rangle $$ such that
$f^*\langle X\rangle =\langle X\times_S T\rangle $ for every
separated $S$-scheme $X$ of finite type. It localizes to a ring
morphism
$$f^*:\mathcal{M}^{mod}_S\rightarrow \mathcal{M}^{mod}_T$$
If $g:S\rightarrow U$ is a separated morphism of finite type
between Noetherian schemes, then there exists a unique morphism of
abelian groups
$$g_{!}:K_0\langle Var_S\rangle \rightarrow K_0\langle Var_U\rangle $$ such that for every
separated $S$-scheme $X$ of finite type, we have $g_!\langle
X\rangle =\langle X|_U\rangle $ (here $X|_U$ denotes the
$U$-scheme obtained by composing the structural morphism
$X\rightarrow S$ with the morphism $g$). Moreover, there exists a
unique morphism of abelian groups
$$g_!:\mathcal{M}^{mod}_S\rightarrow \mathcal{M}^{mod}_U$$
such that
$$g_!(\langle X\rangle \widetilde{\LL}^i_S)=\langle X|_U\rangle \widetilde{\LL}^i_U$$ for
every separated $S$-scheme of finite type $X$ and every integer
$i$.
\subsection{Fibrations}
\begin{lemma}\label{lemm-fibration}
\label{triviale fib rect} Let $S$ be a Noetherian scheme, and let
$X$, $Y$ and $Z$ be separated $S$-schemes of finite type.  Let
$f:X\rightarrow Y$ be a morphism of $S$-schemes, and assume that
for every perfect field $F$ and every morphism of schemes $\Spec
F\rightarrow Y$, there exists a universal homeomorphism of
$F$-schemes
$$X\times_Y \Spec F \rightarrow Z\times_S \Spec F$$
 Then
$$\langle X\rangle =\langle Y\rangle \cdot\langle Z\rangle $$ in $K_0\langle Var_S\rangle $.
  \end{lemma}

\begin{proof}
By Noetherian induction, it is enough to find a non-empty open
subscheme $U$ of $Y$ such that $$\la X\times_Y U\ra =\la U\ra
\cdot \la Z\ra$$ in $K_0\langle Var_S\rangle $. We may assume that
$Y$ is affine and integral. We denote by $B$ the ring of regular
functions on $Y$. Let $F$ be the perfect closure of the function
field $Frac(B)$ of $Y$. We know that there exists a universal
homeomorphism of $F$-schemes
$$g:X\times_Y \Spec F \rightarrow Z\times_S \Spec F$$

The $B$-algebra $F$ is the direct limit of its finitely generated
sub-$B$-algebras. Hence, by \cite[8.8.2 and 8.10.5]{ega4.3}, there
exist a finitely generated sub-$B$-algebra $B'$ of $F$, and a
universal homeomorphism of $B'$-schemes
 $$g':X\times_Y \Spec B' \rightarrow Z\times_S \Spec B' $$  such
 that $g$ is obtained from $g'$ by base change
 from $\Spec B'$ to $\Spec F$.

Since $\Spec B'\rightarrow Y$ is purely inseparable over the
generic point of $Y$, and the generic point of $Y$ is the
projective limit of the dense open subschemes of $Y$, it follows
from \cite[8.10.5]{ega4.3} that there exists a dense open
subscheme $U$ of $Y$ such that $$\Spec B'\times_Y U\rightarrow U$$
is a universal homeomorphism.

Looking at the diagram of universal homeomorphisms of separated
$S$-schemes of finite type
$$\begin{CD}
X\times_Y (\Spec B'\times_Y U) @>g'\times_Y id_U>> Z\times_S
(\Spec B'\times_Y U)
\\ @VVV @VVV
\\ X\times_Y U @. Z\times_S U
\end{CD}$$
we find that
$$\la X\times_Y U\ra =\la U\ra \cdot\la Z\ra$$ in $K_0\langle Var_S\rangle $.
\end{proof}

\subsection{The Grothendieck ring of the theory $AFC_F$.}
The modified Grothendieck ring arises naturally in the setting of
model theory. Let $F$ be a field, and denote by $ACF_F$ the theory
of algebraically closed fields over $F$ in the language
$\mathcal{L}_F$ of $F$-algebras.

Recall that formulas in $\mathcal{L}_F$ consist of quantifiers and
Boolean combinations of polynomial equations with coefficients in
$F$. For every formula $\varphi(x_1,\ldots,x_n)$ in
$\mathcal{L}_F$ and every $F$-algebra $A$, we denote by
$S_{\varphi}(A)$ the subset of $A^n$ defined by $\varphi$.

We say that two formulas $\varphi(x_1,\ldots,x_m)$ and
$\psi(y_1,\ldots,y_n)$ in $\mathcal{L}_F$ are $ACF_F$-equivalent,
if there exists a third formula
$\eta(x_1,\ldots,x_m,y_1,\ldots,y_n)$ such that, for every
algebraically closed field $L$ that contains $F$, the set
$S_{\eta}(L)\subset L^{m+n}$ is the graph of a bijection between
 $S_{\varphi}(L)\subset L^m$ and $S_{\psi}(L)\subset L^n$. We say
 that $\eta$ defines an $ACF_F$-equivalence between $\varphi$ and
 $\psi$.

\begin{definition}
Let $F$ be a field. The Grothendieck group $K_0(ACF_F)$ of the
theory $ACF_F$ is the quotient of the free abelian group on
$ACF_F$-equivalence classes $[\varphi]$ of formulas $\varphi$ in
$\mathcal{L}_F$, by the subgroup generated by elements of the form
$$[\varphi\wedge \psi]+[\varphi\vee \psi]-[\varphi]-[\psi]$$
where $\varphi$ and $\psi$ are formulas in $\mathcal{L}_F$ with
the same sets of free variables.

We endow $K_0(ACF_F)$ with the unique ring structure such that,
for all formulas $\varphi$ and $\psi$ in $\mathcal{L}_F$ in
disjoint sets of free variables, we have
$$[\varphi]\cdot [\psi]=[\varphi\wedge \psi]$$
in $K_0(ACF_F)$.
\end{definition}

The unit for the multiplication in $K_0(ACF_F)$ is the class of
the formula $\psi=(0=0)$. For every $F$-algebra $A$, this formula
defines the set $S_{\psi}(A)=A^0=\{pt\}$.

If $n$ is an element of $\N$, and $i_X:X\rightarrow \A^n_F$ is an
immersion of $F$-schemes, then there exists a formula
$\varphi(x_1,\ldots,x_n)$ such that $$S_{\varphi}(L)=X(L)\subset
L^n$$ for every field $L$ that contains $F$. We call such a
formula $\varphi$ an $i_X$-formula. It is not unique. If $Y$ is a
quasi-affine $F$-variety, then we say that a formula $\psi$ in
$\mathcal{L}_F$ is a $Y$-formula if it is an $i_Y$-formula for
some immersion $i_Y:Y\rightarrow \A^m_F$, with $m\in \N$. Again,
such a $Y$-formula is not unique, but its class $[\psi]$ in
$K_0(ACF_F)$ only depends on $Y$, and not on the choice of the
immersion $i_Y$ or the formula $\psi$.

\begin{lemma}\label{lemm-graph}
Let $F$ be a field. Let $m$ and $n$ be elements of $\N$, and let
$C$ and $D$ be constructible subsets of $\mathbb{A}^m_F$, resp.
$\mathbb{A}^n_F$. Assume that there exists a formula
$\eta(x_1,\ldots,x_m,y_1,\ldots,y_n)$ in $\mathcal{L}_F$ such
that, for every algebraically closed field $L$ containing $F$, the
set
$$S_{\eta}(L)\subset L^{m+n}$$ is the graph of a bijection between
$C(L)\subset L^m$ and $D(L)\subset L^n$. Then we have
$$\langle C\rangle =\langle D\rangle $$ in $K_0\langle Var_F\rangle $.
\end{lemma}
\begin{proof}
By quantifier elimination relative to $ACF_F$, there exists a
unique constructible subset $E$ of $\A^{m+n}_F$ such that
$$S_{\eta}(L)=E(L)\subset L^{m+n}$$ for every algebraically closed field $L$ that contains $F$.
 It suffices to show that $\langle C\rangle =\langle E\rangle $ in $K_0\langle Var_F\rangle $.
  We denote by
$$\pi:\mathbb{A}^{m+n}_F\rightarrow \mathbb{A}^m_F$$ the projection
onto the first $m$ coordinates. By Noetherian induction on $C$,
 it is enough to prove that there exists a non-empty open subset
 $U$ of $C$ such that $U$ is locally closed in $\A^m_F$,
 $V=\pi^{-1}(U)\cap E$ is locally closed in $\A^{m+n}_F$, and the
 morphism
 $V\rightarrow U$ induced by $\pi$ is a universal homeomorphism if we endow $U$
 and $V$ with the reduced induced structure.


 We may assume that $C$ is irreducible and locally
closed in $\mathbb{A}^m_F$, and we endow $C$ with its reduced
induced structure.
 Let $E'$ be an integral
 subscheme of $\mathbb{A}^{m+n}_F$ whose support is contained in $E$ and such that the
morphism $\pi':E'\rightarrow C$ induced by $\pi$ is dominant. Then
$\pi'$ is a quasi-finite morphism of $F$-varieties. Hence, there
exists a non-empty open subset $U$ of $C$ such that, putting
$V=E'\times_C U$, the morphism $V\rightarrow U$ is finite and
surjective. Surjectivity implies, in particular, that $V=E\cap
\pi^{-1}(U)$.
  It
follows from our assumptions that $V\rightarrow U$ is purely
inseparable, so that it is a universal homeomorphism.
\end{proof}

\begin{prop}\label{prop-acf}
Let $F$ be a field. There exists a unique ring morphism
$$\alpha:K_0(Var_F)\rightarrow K_0(ACF_F)$$ such that, for every
 quasi-affine $F$-variety $X$, the image of
$[X]$ under $\alpha$ is the class in $K_0(ACF_F)$ of an
$X$-formula $\varphi_X$.

The morphism $\alpha$ is surjective, and its kernel equals
$I^{uh}_F$, so that $\alpha$ factors through an isomorphism
$$K_0\langle Var_F\rangle \rightarrow K_0(ACF_F)$$
\end{prop}
\begin{proof}
Uniqueness if clear, since the classes of affine $F$-varieties
generate the Grothendieck ring $K_0(Var_F)$. So let us prove the
existence of $\alpha$.

We define the Grothendieck ring $K_0(QAff_F)$ of quasi-affine
$F$-varieties as follows. As an abelian group, $K_0(QAff_F)$ is
the quotient of the free abelian group on isomorphism classes
$[X]'$ of quasi-affine $F$-varieties $X$ by the subgroup generated
by elements of the form $$[X]'-[Y]'-[X\setminus Y]'$$ with $X$ a
quasi-affine $F$-variety and $Y$ a closed subvariety of $X$. We
endow $K_0(QAff_F)$ with the unique ring structure such that for
all quasi-affine $F$-varieties $X$ and $Y$, we have
$$[X]'\cdot[Y]'=[(X\times_F Y)_{red}]'$$

It follows easily from the scissor relations in the Grothendieck
ring that there is a unique morphism of abelian groups
$$\beta:K_0(QAff_F)\rightarrow K_0(Var_F)$$ that maps $[X]'$ to
$[X]$ for every quasi-affine $F$-variety $X$, and that $\beta$ is
an isomorphism of rings.
 It is also straightforward to check that there exists a unique ring
morphism
$$\alpha':K_0(QAff_F)\rightarrow K_0(ACF_F)$$ that maps $[X]'$ to
$[\varphi_X]$ for every quasi-affine $F$-variety $X$, where
$\varphi_X$ is an $X$-formula. We can define $\alpha$ as
$\alpha'\circ \beta^{-1}$.

Now, we show that $ker(\alpha)$ contains the ideal $I^{uh}_F$. Let
$f:X\rightarrow Y$ be a universal homeomorphism of $F$-varieties.
We choose a partition $\{Y_1,\ldots,Y_r\}$ of $Y$ into locally
closed subsets, and we endow $Y_i$ with its reduced induced
structure, for each $i$. We may assume that $Y_i$ is affine for
every $i$. If we put $X_i=(X\times_Y Y_i)_{red}$, then the
morphism $X_i\rightarrow Y_i$ is a universal homeomorphism, $X_i$
is affine, and we have
\begin{eqnarray*}
 [X] &=& \sum_{i=1}^r [X_i]
\\ {[}Y] &=& \sum_{i=1}^r [Y_i]
\end{eqnarray*}
in $K_0(Var_F)$. So it is enough to prove that $\alpha([X]-[Y])=0$
if $f:X\rightarrow Y$ is a universal homeomorphism of affine
$F$-varieties. We denote by $\Gamma_f$ the graph of $f$ in
$X\times_F Y$, and we choose closed immersions $i_X:X\rightarrow
\A^m_F$ and $i_Y:Y\rightarrow \A^n_F$, with $m,\,n\in \N$. These
induce a closed immersion $i_f:\Gamma_f\rightarrow \A^{m+n}_F$. If
we choose an $i_X$- formula $\varphi_{X}$, an $i_Y$-formula
$\varphi_Y$ and an $i_f$-formula $\varphi_{f}$, then $\varphi_{f}$
defines an $ACF_F$-equivalence between $\varphi_X$ and
$\varphi_Y$, so that $\alpha([X])=\alpha([Y])$.

Hence, the morphism $\alpha$ factors through a ring morphism
$$\gamma:K_0\langle Var_F\rangle \rightarrow K_0(ACF_F)$$ We'll show that it
is an isomorphism by constructing its inverse. If
$\varphi(x_1,\ldots,x_m)$ is a formula in $\mathcal{L}_F$, then by
quantifier elimination relative to $ACF_F$,
 there exists a unique constructible subset $C_{\varphi}$ of
$\A^m_F$ such that, for every algebraically closed field $L$ that
contains $F$, the subsets $S_{\varphi}(L)$ and $C_{\varphi}(L)$ of
$L^m$ coincide. It follows from Lemma \ref{lemm-graph} that the
class $\langle C_{\varphi}\rangle $ of $C_{\varphi}$ in
$K_0\langle Var_F\rangle $ only depends on the $ACF_F$-equivalence
class of $\varphi$. It is easily seen that there exists a unique
ring morphism
$$\delta:K_0(ACF_F)\rightarrow K_0\langle Var_F\rangle $$ that maps
$[\varphi]$ to $[C_{\varphi}]$ for every formula $\varphi$ in
$\mathcal{L}_F$. This ring morphism is inverse to $\gamma$.
\end{proof}

\section{Compatibility with realization morphisms}\label{sec-real}
In this section, we'll prove that the realization morphisms from
Section \ref{sec-k0} factor through the modified Grothendieck
ring.

\begin{prop}\label{prop-real}
Let $S$ be a Noetherian scheme, and let $X$ and $Y$ be separated
$S$-schemes of finite type such that there exists a universal
homeomorphism of $S$-schemes
$$f:X\rightarrow Y$$
We denote by $g_X$ and $g_Y$ the structural morphisms from $X$,
resp. $Y$, to $S$.

\begin{enumerate}
\item If $S$ is the spectrum of a finite field $F$, then $X(F)$
and $Y(F)$ have the same cardinality. In particular, the point
counting realization
$$\sharp:K_0(Var_F)\rightarrow \Z$$
factors through a ring morphism
$$\sharp:K_0\langle Var_F\rangle \rightarrow \Z$$

\item Let $\ell$ be a prime invertible on $S$. The natural
morphism
 $$R(g_Y)_!(\Q_\ell)\rightarrow R(g_X)_!(\Q_\ell)$$ in $D^b_c(S,\Q_\ell)$ induced by $f$ is an
 isomorphism. In particular, the \'etale
 realization
 $$\acute{e}t:K_0(Var_S)\rightarrow K_0(D^b_c(S,\Q_\ell))$$
factors through a ring morphism
 $$\acute{e}t:K_0\langle Var_S\rangle \rightarrow K_0(D^b_c(S,\Q_\ell))$$
 which localizes to a ring morphism
  $$\acute{e}t:\mathcal{M}^{mod}_S\rightarrow K_0(D^b_c(S,\Q_\ell))$$

\item If $F$ is a field, and $\ell$ a prime number invertible in
$F$, then the Galois realization
$$Gal:K_0(Var_F)\rightarrow K_0(Rep_{G_F}\Q_\ell)$$
factors through a ring morphism
$$Gal:K_0\langle Var_F\rangle \rightarrow K_0(Rep_{G_F}\Q_\ell)$$
which localizes to a ring morphism
$$Gal:\mathcal{M}^{mod}_F \rightarrow K_0(Rep_{G_F}\Q_\ell)$$

 \item The Poincar\'e polynomials $P(g_X):S\rightarrow \Z[T]$ and
 $P(g_Y):S\rightarrow \Z[T]$ are equal. In particular, the
 Poincar\'e realization
$$P_S:K_0(Var_S)\rightarrow \mathcal{C}(S,\Z[T])$$ factors through
a ring morphism
$$P_S:K_0\langle Var_S\rangle \rightarrow \mathcal{C}(S,\Z[T])$$
which localizes to a ring morphism
$$P_S:\mathcal{M}^{mod}_S\rightarrow \mathcal{C}(S,\Z[T,T^{-1}])$$
 \end{enumerate}
\end{prop}
\begin{proof}
(1) Since $f$ is a universal homeomorphism and $F$ is perfect, the
map
$$f(F):X(F)\rightarrow Y(F)$$ is a bijection.

(2) By Grothendieck's spectral sequence for the composition of the
functors $f_{!}=f_*$ and $(g_{Y})_!$, it suffices to show that
\begin{eqnarray}
f_{*}\Q_\ell&\cong& \Q_\ell, \label{eq-f0}
\\
 R^jf_{*}(\Q_\ell)&=&0 \mbox{ for } j>0. \label{eq-fj}
 \end{eqnarray} The isomorphism (\ref{eq-f0}) follows immediately
 from the fact that $f$ is a universal homeomorphism. The equality
 (\ref{eq-fj}) follows from finiteness of
 $f$.

(3) This is a special case of point (2).

 (4) By definition of the Poincar\'e polynomial \cite[8.12]{Ni-tracevar}, it is enough to consider the case where $S$ is the spectrum of
 a field $F$ of characteristic $p>0$. By \cite[8.8.2 and 8.10.5]{ega4.3}, there exist a finitely
 generated sub-$\mathbb{F}_p$-algebra $A$ of $F$, and a universal homeomorphism
 $$f':X'\rightarrow Y'$$ of separated $A$-schemes of finite type, such
 that $f:X\rightarrow Y$ is obtained from $f'$ by base change
 from $\Spec A$ to $S=\Spec F$. We denote by $g_{X'}$ and $g_{Y'}$
 the structural morphisms from $X'$ and $Y'$ to $\Spec A$.

 If we denote by $\eta$ the generic
 point of $\Spec A$ and by $\kappa(\eta)$ its residue field, then
 the Poincar\'e polynomial $P(g_X)$, resp. $P(g_Y)$, is equal to the Poincar\'e polynomial of the separated
 $\kappa(\eta)$-scheme of finite type $X'\times_A \kappa(\eta)$,
 resp. $Y'\times_A \kappa(\eta)$ \cite[8.12]{Ni-tracevar}. Since $P(g_{X'})$ and
 $P(g_{Y'})$ are constructible functions on $\Spec A$ \cite[8.12]{Ni-tracevar}, it suffices
 to show that their values coincide at all closed points of $\Spec A$.
 Hence, we may assume that $F$ is a finite field. In this case,
 the result follows immediately from point (2), by definition of the Poincar\'e polynomial of a variety over a finite field \cite[8.1]{Ni-tracevar}.
\end{proof}
\begin{cor}\label{cor-euler}
If $F$ is a field, then the \'etale Euler characteristic
$$\chi_{top}:K_0(Var_F)\rightarrow \Z$$ factors through a ring morphism
$$\chi_{top}:K_0\langle Var_F\rangle \rightarrow \Z$$ which
localizes to a ring morphism
$$\chi_{top}:\mathcal{M}^{mod}_F \rightarrow \Z$$
\end{cor}

\section{Motivic integration in mixed
characteristic}\label{sec-mixed} In this section,
 we will fill a gap in the proof of the change of
variables theorem for motivic integrals on formal schemes in mixed
characteristic \cite[8.0.5]{sebag1}. To this aim, it is necessary
to replace the localized Grothendieck ring of varieties by the
modified version introduced in Definition \ref{def-mod}. We
emphasize that this correction only affects the theory of motivic
integration in {\em mixed} characteristic; in equal characteristic
$p\geq 0$, the results in
 \cite[7.1.3 and 8.0.5]{sebag1} are valid as stated.

Throughout this section, we denote by $R$ a complete discrete
valuation ring of mixed characteristic with quotient field $K$ and
perfect residue field $k$. We put $\mathbb{D}=\Spf\,R$. An $stft$
formal $R$-scheme is a separated formal $R$-scheme topologically
of finite type.

\subsection{The change of variables formula}
We consider a morphism of formal $R$-schemes $h:Y\rightarrow X$,
where $X$ and $Y$ are flat $stft$ formal $R$-schemes of pure
relative dimension $d$. The proof of the change of variables
theorem \cite[8.0.5]{sebag1} is based on the following lemma
\cite[7.1.3]{sebag1}. For terminology and notations, we refer to
\cite{sebag1}.
\begin{lem} [\cite{sebag1}, Lemme 7.1.3]\label{lemm-err}
\label{lem a pbs} Assume that the structural morphism
$Y\rightarrow\mathbb{D}$ is smooth. Let $B\subset Gr(Y)$ be an
$m$-cylinder, for some $m\in \N$, and put $A=h(B)$. Assume that
$e$ and $e'$ are elements of $\N$ such that
$\mathrm{ord}_{\pi}(\mathrm{Jac})_h(\varphi)=e$
 for all $\varphi\in B$, and
$A\subset Gr^{(e')}(X)$. Then $A$ is a cylinder.

Moreover, if the restriction of $h$ to $B$ is injective, then, for
all integers $n\geq \mathrm{max}(2e+c_X, m+e)$, the following
properties hold.

\begin{enumerate}

\item If $\varphi$ and $\varphi'$ belong to $B$ and
$\pi_{n,X}(h(\varphi))=\pi_{n,X}(h(\varphi'))$, then
$\pi_{n-e,Y}(\varphi)=\pi_{n-e,Y}(\varphi')$,

\item We have $[\pi_{n,Y}(B)]=[\pi_{n,X}(A)]\LL^e$ in
$K_0(Var_k)[\LL^{-1}]$.
\end{enumerate}
\end{lem}

Recall that, for every $n\in \N$ and every $stft$ formal
$R$-scheme $Z$, $Gr_n(Z)$ denotes the level $n$ Greenberg scheme
of $Z$ \cite[\S\,3]{sebag1}. The proof of point (2) in
\cite[7.1.3]{sebag1} only computes the fibers
$$Gr_{n}(Y)\times_{Gr_n(X)}\Spec F$$ when $F$ is a {\em perfect}
field containing $k$ and $\Spec F\rightarrow Gr_n(X)$ a morphism
of $k$-schemes. Indeed, only when $F$ is perfect, we can identify
the set $Gr(X)(F)$ with the set $X(R_F)$, where
$$R_F=R\widehat{\otimes}_{W(k)}W(F)$$ is a complete discrete
valuation ring that is an unramified extension of $R$. Then the
proof of \cite[7.1.3]{sebag1} shows that there exists an
isomorphism of $F$-schemes $$Gr_{n}(Y)\times_{Gr_n(X)}\Spec F\cong
\A^e_F$$ so that we can deduce from Lemma \ref{lemm-fibration}
that $$\la \pi_{n,Y}(B)\ra=\la
\pi_{n,X}(A)\ra\widetilde{\LL}^e_k$$ in $K_0\la Var_k\ra$.
However, we cannot conclude that the equality in point (2) of
Lemma \ref{lemm-err} holds. Therefore, we have to replace Lemma
\ref{lemm-err} by the following statement.

\begin{lem} [Corrected form] \label{lemm-corr}
 Assume that the structural morphism
$Y\rightarrow\mathbb{D}$ is smooth. Let $B\subset Gr(Y)$ be an
$m$-cylinder, for some $m\in \N$, and put $A=h(B)$. Assume that
$e$ and $e'$ are elements of $\N$ such that
$\mathrm{ord}_{\pi}(\mathrm{Jac})_h(\varphi)=e$
 for all $\varphi\in B$, and
$A\subset Gr^{(e')}(X)$. Then $A$ is a cylinder.

Moreover, if the restriction of $h$ to $B$ is injective, then, for
all integers $n\geq \mathrm{max}(2e+c_X, m+e)$, the following
properties hold.

\begin{enumerate}

\item If $\varphi$ and $\varphi'$ belong to $B$ and
$\pi_{n,X}(h(\varphi))=\pi_{n,X}(h(\varphi'))$, then
$\pi_{n-e,Y}(\varphi)=\pi_{n-e,Y}(\varphi')$,

\item We have $\la
\pi_{n,Y}(B)\ra=\la\pi_{n,X}(A)\ra\widetilde{\LL}^e$ in $K_0\la
Var_k\ra$.
\end{enumerate}
\end{lem}

The statements in \cite[8.0.3]{sebag1} and \cite[8.0.5]{sebag1}
(change of variables theorem) have to be modified accordingly,
replacing the completed Grothendieck ring
$\widehat{\mathcal{M}}_k$ by the completed modified Grothendieck
ring $\widehat{\mathcal{M}}^{mod}_k$.

We emphasize that we do not know a counter-example to the original
statement in Lemma \ref{lemm-err}, since we do not know if the
projection morphism
$$K_0(Var_k)\rightarrow K_0\la Var_k\ra$$ is an isomorphism.

\subsection{Integration on rigid varieties and the motivic Serre
invariant}\label{sec-rig} If $R$ has mixed characteristic, the
modification of the change of variables theorem in
\cite[8.0.5]{sebag1} also affects the theory of motivic
integration on rigid $K$-varieties developed in \cite{motrigid}.
Theorem-Definition 4.1.2 in \cite{motrigid} should be replaced by
the following statement ({\em only} in the mixed characteristic
case; in equal characteristic $p\geq 0$, all the results in
\cite{motrigid} are valid as stated). For terminology and
notations, we refer to \cite{motrigid}.

\begin{theorem-definition}
Let $X$ be a smooth separated quasi-compact rigid variety over
$K$, of pure dimension $d$. Let $\omega$ be a differential form in
$\Omega^d_{X} (X)$.

\smallskip

(1) Let $\cX$ be a $stft$ formal $R$-model of $X$. Then the
function $\mathrm{ord}_{\varpi, \cX} (\omega)$ is exponentially
integrable on $Gr (\cX)$ and the image of the motivic integral
$$\int_{Gr (\cX)} \LL^{- \mathrm{ord}_{\varpi, \cX}
(\omega)} d\mu$$ in $\widehat{\mathcal{M}}^{mod}_k$ does not
depend on the model $\cX$. We denote it by $\int_X \omega d\mu$.

\smallskip

(2) Assume moreover that $\omega$ is a gauge form, i.e., that it
generates $\Omega^d_{X}$ at every point of $X$, and assume
 that some open formal subscheme  $\cU$ of $\cX$ is a weak
N{\'e}ron model of $X$. Then the function $\mathrm{ord}_{\varpi,
\cX} (\omega)$ takes only a finite number of values on
$Gr(\mathcal{X})$, and its fibres are stable cylinders. The image
of the motivic integral
$$\int_{Gr (\cX)} \LL^{- \mathrm{ord}_{\varpi, \cX}
(\omega)} d\tilde \mu$$ in $\mathcal{M}_k^{mod}$ does not depend
on the model $\cX$. We denote it by $\int_X \omega d \tilde \mu$.
\end{theorem-definition}

If $R$ has mixed characteristic, the definition of the motivic
Serre invariant \cite[4.5.1 and 4.5.3]{motrigid} has to be
modified accordingly.
\begin{theorem-definition}
Let $X$ be a smooth separated quasi-compact rigid $K$-variety, and
let $\mathscr{U}$ be a weak N\'eron model of $X$. The class
$$\la\mathscr{U}_s\ra\in K_0\la Var_k \ra/(\widetilde{\LL}_k-1)$$
of the special fiber $\mathscr{U}_s$ of $\mathscr{U}$ only depends
on $X$, and not on the choice of weak N\'eron model. We denote it
by $S(X)$, and we call it the motivic Serre invariant of $X$.
\end{theorem-definition}


\subsection{Further corrections to the literature}\label{sec-lit}
The theory of motivic integration on formal schemes and rigid
varieties in mixed characteristic has been applied in several
other articles. All of them can  easily be corrected, replacing
the Grothendieck ring of varieties by the modified Grothendieck
ring of varieties. This is harmless for the applications of the
theory, since all the realization morphisms that are used factor
through the modified Grothendieck ring (see Section
\ref{sec-real}). Let us indicate some of the changes that should
be made ({\em only} in the mixed characteristic case; in equal
characteristic $p\geq 0$, all the results are valid as stated).

\bigskip
 In
\cite[4.18]{NiSe-curves}, the last two lines of the statement
should be replaced by: ``$\ldots\la \pi_n(B)\ra = \la
\pi_n(A)\ra\widetilde{\LL}^e_{X_s}$ in $K_0\la Var_{X_s}\ra$''. In
\cite[4.19]{NiSe-curves}, the ring $\widehat{\mathcal{M}}_{X_s}$
should be replaced by $\widehat{\mathcal{M}}^{mod}_{X_s}$, and in
\cite[4.20(2)]{NiSe-curves}, the ring $\mathcal{M}_{X_s}$ should
be replaced by $\mathcal{M}^{mod}_{X_s}$. Likewise, in
\cite[\S\,6]{NiSe-curves}, the Grothendieck rings have to be
replaced by their modified analogues. In particular, the motivic
Serre invariant of a generically smooth $stft$ formal $R$-scheme
$X_{\infty}$ of pure relative dimension \cite[6.2]{NiSe-curves} is
well-defined in $K_0\la Var_{X_s}\ra/(\widetilde{\LL}_{X_s}-\la
X_s\ra)$.

\bigskip
In  \cite[\S\,3.2]{NiSe}, the various Grothendieck rings
should be replaced by the modified Grothendieck rings. The trace
formula in \cite[5.4]{NiSe} remains valid, because the $\ell$-adic
Euler characteristic factors through the modified Grothendieck
ring (Corollary \ref{cor-euler}).

\bigskip
 In
\cite[\S\,5]{NiSe-weilres}, the various Grothendieck rings should
be replaced by the modified Grothendieck rings.

\bigskip
 In Sections 4 and 5.3 of \cite{ni-trace}, the various Grothendieck rings should
be replaced by the modified Grothendieck rings. The trace formula
in \cite[6.4]{ni-trace} remains valid, because the $\ell$-adic
Euler characteristic factors through the modified Grothendieck
ring (Corollary \ref{cor-euler}).

\bigskip
 In \cite[\S\,5]{Ni-tracevar}, in particular in Theorem 5.4 and Definition 5.5, the motivic Serre
invariant should take its values in $K_0\la Var_k\ra
/(\widetilde{\LL}_k-1)$. The results in Sections 6 and 7 of
\cite{Ni-tracevar} remain valid,  because the Poincar\'e
polynomial and the $\ell$-adic Euler characteristic factor through
the modified Grothendieck ring (Proposition \ref{prop-real}(4) and
Corollary \ref{cor-euler}).


\begin{thebibliography}{10}

\bibitem{ega4.2}
A.~Grothendieck and J.~Dieudonn\'e.
\newblock {El\'ements de {G}\'eom\'etrie {A}lg\'ebrique, IV, Deuxi\`eme
  partie.}
\newblock {\em Publ. Math., Inst. Hautes \'Etud. Sci.}, 24:5--231, 1965.

\bibitem{ega4.3}
A.~Grothendieck and J.~Dieudonn\'e.
\newblock {El\'ements de {G}\'eom\'etrie {A}lg\'ebrique, IV, Troisi\`eme
  partie.}
\newblock {\em Publ. Math., Inst. Hautes \'Etud. Sci.}, 28:5--255, 1966.


\bibitem{Larsen-Lunts}
M.~Larsen and V.~A. Lunts.
\newblock {Motivic measures and stable birational geometry.}
\newblock {\em Mosc. Math. J.}, 3(1):85--95, 2003.

\bibitem{Liu-Sebag}
Q.~Liu and J.~Sebag.
\newblock {The Grothendieck ring of varieties and piecewise isomorphisms}.
\newblock {\em to appear in Math. Z.}

\bibitem{motrigid}
F.~Loeser and J.~Sebag.
\newblock {Motivic integration on smooth rigid varieties and invariants of
  degenerations}.
\newblock {\em Duke Math. J.}, 119:315--344, 2003.

\bibitem{ni-trace}
J.~Nicaise.
\newblock A trace formula for rigid varieties, and motivic Weil generating
  series for formal schemes.
\newblock {\em Math. Ann.}, 343(2):285--349, 2009.

\bibitem{Ni-tracevar}
J.~Nicaise.
\newblock {A trace formula for varieties over a discretely valued field}.
\newblock {\em to appear in J. Reine Angew. Math.}, arxiv:0805.1323.

\bibitem{NiSe-curves}
J.~Nicaise and J.~Sebag.
\newblock{\em Motivic Serre invariants of curves}.
\newblock{Manuscr. Math.} 123(2):105--132, 2007.

\bibitem{NiSe}
J.~Nicaise and J.~Sebag.
\newblock The motivic {S}erre invariant, ramification, and the analytic
  {M}ilnor fiber.
\newblock {\em Invent. Math.}, 168(1):133--173, 2007.

\bibitem{NiSe-weilres}
J.~Nicaise and J.~Sebag.
\newblock Motivic {S}erre invariants and {W}eil restriction.
\newblock {\em J. Algebra}, 319(4):1585--1610, 2008.

\bibitem{sebag1}
J.~Sebag.
\newblock Int\'egration motivique sur les sch\'emas formels.
\newblock {\em Bull. Soc. Math. France}, 132(1):1--54, 2004.

\bibitem{Sebag-PAMS}
J.~Sebag.
\newblock Variations on a question of Larsen and Lunts.
\newblock{\em to appear in Proc. Am. Math. Soc.}
\end{thebibliography}
\end{document}